\documentclass[a4paper,11pt]{amsart}
\usepackage[T1]{fontenc}
\usepackage{amsmath}
\usepackage{amsthm}
\usepackage{amssymb}
\usepackage{mathrsfs}
\usepackage[margin=1.2in]{geometry}
\usepackage{graphicx}
\usepackage[pdftex,%
bookmarksopen=false,
pdfpagemode=none 
]{hyperref}

\def\O{\mathcal{O}}

\theoremstyle{plain}
\makeatletter
\newtheorem*{rep@theorem}{\rep@title}
\newcommand{\newreptheorem}[2]{%
\newenvironment{rep#1}[1]{%
 \def\rep@title{#2 \ref{##1}}%
 \begin{rep@theorem}}%
 {\end{rep@theorem}}}
\makeatother

\newtheorem{theorem}{Theorem}
\newreptheorem{theorem}{Theorem}
\newtheorem{lemma}[theorem]{Lemma}
\newtheorem{claim}{Claim}
\newtheorem{definition}{Definition}
\newtheorem{observation}{Main Observation}
\newreptheorem{lemma}{Lemma}

\title{Transfinite Ford-Fulkerson on a Finite Network}
\author{Spencer Backman}
\address[Spencer Backman]{Department of Computer Science, University of Rome ``La Sapienza''}
\email{backman@di.uniroma1.it}

\author{Tony Huynh}
\address[Tony Huynh]{Department of Computer Science, University of Rome ``La Sapienza''}
\email{tony.bourbaki@gmail.com}

\thanks{This research was supported by the European Research Council under the European Unions Seventh Framework
Programme (FP7/2007-2013)/ERC Grant Agreement no. 279558.}

\begin{document}

\begin{abstract}
It is well-known that the Ford-Fulkerson algorithm for finding a maximum flow in a network need  not terminate if we allow the arc capacities to take irrational values.  Every non-terminating example converges to a limit flow, but this limit flow need not be a maximum flow.  Hence, one may pass to the limit and begin the algorithm again. 
In this way, we may view the Ford-Fulkerson algorithm as a transfinite algorithm.  

We analyze the transfinite running-time of the Ford-Fulkerson algorithm using ordinal numbers, and prove that the worst case running-time is $\omega^{\Theta(|E|)}$.
For the lower bound, we show that we can model the Euclidean algorithm via Ford-Fulkerson on an auxiliary network.  By running this example on a pair of incommensurable numbers, we obtain a new robust non-terminating example.  We  then describe how to glue $k$ copies of our Euclidean example in parallel to obtain running-time $\omega^k$.  An upper bound of $\omega^{|E|}$ is established via induction on $|E|$.  
We conclude by illustrating a close connection to transfinite chip-firing as previously investigated by the first author \cite{backman2014infinite}.
\end{abstract}

\maketitle

\section{Introduction}
The Ford-Fulkerson algorithm \cite{fordfulkerson} is a classic algorithm for computing the maximum flow in a network.  At each step, the algorithm finds
an \emph{augmenting path} $P$ from the source vertex $s$ to the sink vertex $t$ and then pushes as much flow as possible along $P$.  This procedure is then iterated, until no such augmenting path exists.  
It is well-known that in certain networks with irrational capacities, the Ford-Fulkerson algorithm does not necessarily terminate if the augmenting paths are not chosen carefully.  The smallest non-terminating example is due to Zwick \cite{nonterminating}.

Dinits \cite{dinits} and Edmonds and Karp \cite{edmondskarp} independently showed that if one always chooses an augmenting path of minimum length, then the Ford-Fulkerson algorithm will necessarily terminate.  We warn the reader that whenever we refer to the Ford-Fulkerson algorithm, we 
mean the original version where augmenting paths can be chosen arbitrarily.  

It is fairly easy to show that if the Ford-Fulkerson algorithm does not terminate, then it will converge to a (not necessarily maximum) flow $f$.  Thus, after  
$\omega$ steps, we may begin the algorithm anew, starting with the limit flow $f$.  By iterating this procedure, we can view the Ford-Fulkerson algorithm as
a transfinite algorithm and ask what its worst case running-time is in terms of ordinal numbers.  Note that the notion of using ordinals as a complexity measure dates back at least to the work of Turing \cite{turing}.

The following theorem, which is the main result of this paper, determines this worst case ordinal running-time
up to a constant factor in the exponent. 

\begin{theorem} 
The worst case running-time of the Ford-Fulkerson algorithm is $\omega^{\Theta(|E(N)|)}$.
\end{theorem}

This theorem is established via the following two lemmas.

\begin{lemma} \label{upper}
For every network $N$, every run of the Ford-Fulkerson algorithm terminates after at most $\omega^{|E(N)|}$ steps.
\end{lemma}

The proof of Lemma \ref{upper} is by induction.  Although we are working with ordinal numbers where transfinite induction might seem like a natural tool, the argument proceeds by finite induction on the exponent $|E(N)|$.

\begin{lemma} \label{lower}
For every $\ell \in \mathbb{N}$, there exists a network $N(\ell)$ on $\ell$ arcs and a run of the Ford-Fulkerson algorithm on $N(\ell)$ with run-time at least $\omega^{\left \lfloor{\ell \over  25}\right \rfloor}$. 
\end{lemma}

To prove Lemma \ref{lower}, we first construct a new non-terminating example of Ford-Fulkerson.  The main idea is to demonstrate that the Euclidean algorithm can be modelled by applying Ford-Fulkerson to a particular network.  Thus, a run of the Euclidean algorithm on two incommensurable numbers gives a non-terminating example of Ford-Fulkerson.  To obtain the general result, we demonstrate how to glue several copies of the Euclidean example in parallel and run them lexicographically. 

In \cite{backman2014infinite}, the first author investigates a certain transfinite chip-firing process on metric graphs.  In the last section we describe a close connection between the results presented in this article and those appearing in \cite{backman2014infinite}.

\section{Ordinal Running-time}
We now describe how we intend to measure running-time via ordinal numbers.  For an introduction to ordinals and network flow theory, we refer the reader to \cite{settheory} and \cite{combopt}, respectively. Roughly speaking, we use an extended notion of a Turing machine that can complete an infinite number of steps of computation, and continue computing afterwards.  This matches the notion of \emph{infinite time Turing machines} by Hamkins and Lewis \cite{infiniteturing}. However, as a tradeoff, we allow a saboteur to choose the augmenting paths at every step of the algorithm.  The ordinal running-time is then the worst running-time over all possible sets of choices of the saboteur.  We give the precise details below.   

All networks considered will always be finite with finite arc capacities.  
Each step of the Ford-Fulkerson algorithm will be indexed by an ordinal $\alpha$ and the corresponding flow after step $\alpha$ will be denoted $f_\alpha$ (we allow $f_0$ be any valid flow).  A \emph{run} of the algorithm is obtained by choosing 
an augmenting path $P_{\beta+1}$ for $f_\beta$ (if it exists) for each successor ordinal $\beta+1$, and setting $f_{\beta+1}$ to be the flow obtained from $f_{\beta}$ by pushing
as much flow as possible along $P_{\beta+1}$.  If no augmenting path exists at step $\beta+1$, we define $f_{\beta+1}$ to be $f_{\beta}$.  If $\alpha$ is a limit ordinal, we define $f_{\alpha}$ to be 
a certain limit flow.  Some care must be taken to ensure that this limit flow is well-defined and this is the content of Lemma \ref{well-defined}.  

The \emph{run-time} of a particular run is the least ordinal $\alpha$ such that $f_{\alpha}=f_{\alpha+1}$.  For a fixed network $N$, the (worst case) \emph{running-time} of the Ford-Fulkerson algorithm 
on $N$ is the maximum of the run-times over all runs of the algorithm on $N$.  
The main result of this section is that transfinite Ford-Fulkerson is a well-defined procedure.  

\begin{lemma} \label{well-defined}
For every ordinal $\alpha$, every run of the Ford-Fulkerson algorithm assigns a well-defined flow $f_\alpha$ after step $\alpha$.
\end{lemma}

\begin{proof}
Let $N=(V,E)$ be a network with a source $s \in V$, a sink $t \in V$, and non-negative finite capacities $c(u,v)$ for each arc $(u,v) \in E$.  We proceed by transfinite induction.  We let $f_0$ be any valid initial flow.  Now, let $\alpha$ be an ordinal, and assume that $f_\beta$ is defined for all ordinals
$\beta < \alpha$. 

First suppose $\alpha$ is a successor ordinal, say $\alpha=\beta+1$.  If there is no augmenting path for $f_{\beta}$, then we set $f_\alpha:=f_\beta$. Otherwise, if the saboteur chooses the augmenting path $P_{\beta+1}$ for $f_\beta$, then we define $f_{\alpha}$ to be the flow obtained from $f_{\beta}$ by pushing as much as possible along $P_{\beta+1}$.  

If $\alpha$ is a limit ordinal, we proceed as follows.  For each arc $e$ and ordinal $\beta < \alpha$, we define a $\{-1,0, 1\}$-valued variable $y(e,\beta)$ as follows.   If $\beta:=\gamma+1$ and the saboteur chose $P_{\gamma+1}$ as the augmenting path for $f_\gamma$, then we set $y(e, \beta)$ to be 1 if $e$ is a forward arc of $P_{\beta}$, -1 if $e$ is a backward arc of $P_{\beta}$, and 0 otherwise.  We initialize $y(e,0):=0$.  We also let $x_0$ be the value of the initial flow $f_0$ and set $x_{\beta}$ to be the amount of flow pushed by $P_\beta$.  If $\beta$ is a limit ordinal or there is no augmenting path at step $\beta$, we set both $y(e, \beta)$ and $x_{\beta}$ to be 0.  

Consider $\sum_{\beta < \alpha} x_{\beta}$.  Observe that at most countably many terms are non-zero, since this sum is bounded (by the capacity of a minimum cut).  Moreover, this series converges absolutely, since each term is non-negative.  Therefore, this sum is independent of the order of summation and is hence well-defined.  We define a flow $f_{\alpha}$ by setting 
\[
f_{\alpha} (e)=f_0(e)+ \sum_{\beta < \alpha} y(e, \beta) x_{\beta}
\]
for each $e \in E$.  Observe that $\sum_{\beta < \alpha} y(e, \beta) x_{\beta}$ is an absolutely convergent series, since $\sum_{\beta < \alpha} x_{\beta}$ converges.  Therefore,  $\sum_{\beta < \alpha} y(e, \beta) x_{\beta}$ is also a well-defined sum.  Evidently, $0 \leq f_{\alpha}(u,v) \leq c(u,v)$ for all $(u,v) \in E$ since $0 \leq f_{\beta}(u,v) \leq c(u,v)$ for all $\beta < \alpha$.  
It remains to verify that $f_{\alpha}$ satisfies conservation of flow.  Let $u \in V \setminus \{s,t\}$.  Observe that 

\begin{align*}
\sum_{(v,u) \in E} f_{\alpha}(v,u) - \sum_{(u,v) \in E} f_{\alpha}(u,v) &= \sum_{(v,u) \in E} \lim_{\beta \to \alpha} f_{\beta}(v,u) - \sum_{(u,v) \in E}  \lim_{\beta \to \alpha} f_{\beta}(u,v) 
\\ &= \lim_{\beta \to \alpha} \Big( \sum_{(v,u) \in E} f_{\beta}(v,u) - \sum_{(u,v) \in E} f_{\beta}(u,v) \Big)
\\ &= 0
\end{align*}
where the last equality follows from conservation of flow for $f_{\beta}$.
\end{proof}

\section{The Upperbound}
Given a flow $f$, an arc $e$ is a \emph{zero-arc} if $f(e)=0$, is \emph{saturated} if $f(e)$ is equal to the capacity of $e$, and is \emph{extreme} if it is saturated or a zero-arc.   

\begin{lemma} \label{extremearcs}
For every $k \in \mathbb{N}$, every network $N$, and every run of the Ford-Fulkerson algorithm on $N$, either the algorithm has already terminated after $\omega^{k}$ steps or there are at least $k+1$ extreme arcs in $f_{\omega^{k}}$
\end{lemma}

\begin{proof}
We proceed by (finite) induction on $k$.  Since some arc must be extreme after pushing as much flow along an augmenting path, the lemma clearly holds for $k=0$ (note $\omega^{0}=1$).  We inductively assume that the claim holds for $k$.  Now, for each $j \in \mathbb{N}$, let $A_j$ be the set of extreme arcs for the flow $f_{j\omega^{k}}$.  Since $\omega^{k}$ steps have passed between $j\omega^{k}$ and $(j+1)\omega^{k}$, by induction we may assume that $|A_j| \geq k+1$ for all $j$.  Next note that each $e \in A$ can only switch between being a zero-arc and a saturated arc a finite number of times because the value of the flow is always bounded.  This implies that for all sufficiently large $j$, there is an arc $a_j \notin A_j$ such that $a_j$ is extreme in $f_{j \omega^{k} + 1}$. Let $A$ be such that $A_j=A$ infinitely often and $a$ be such that $a_j=a$ infinitely often.  Since $\lim_{\beta \to \omega^{k+1}} f_{\beta}=\lim_{j \to \omega} f_{j \omega^{k}}$, it follows that each edge $e \in A$ is extreme for $f_{\omega^{k+1}}$.  On the other hand, since $f_{\omega^{k+1}}=\lim_{j \to \omega} f_{j \omega^{k} + 1}$, we conclude that $a$ is also extreme for $f_{\omega^{k+1}}$. Thus, $f_{\omega^{k+1}}$ has at least $k+2$ extreme arcs, as required.  
\end{proof}

Using Lemma \ref{extremearcs}, we now prove our upperbound, restated for convenience.

\begin{replemma}{upper}
For every network $N$, every run of the Ford-Fulkerson algorithm terminates after at most $\omega^{|E(N)|}$ steps.
\end{replemma}

\begin{proof}
If the algorithm has not terminated after $\omega^{|E(N)|}$ steps, then Lemma \ref{extremearcs} implies that for all  $j \in \mathbb{N}$, every arc of $N$ is extreme in $f_{j \omega^{|E(N)|-1}}$.  Let $c$ be the smallest non-zero capacity over all arcs of $N$.  Since all arcs are extreme in $f_{j \omega^{|E(N)|-1}}$, the flow increases by at least $c$ at step $j \omega^{|E(N)|-1}+1$ (for all $j$).  This is a contradiction since the value of $f_{\omega^{|E(N)|}}$ is finite. 
\end{proof}

\section{The Lowerbound} \label{sec:lowerbound}

Let $a, b \in {\mathbb R}_+$.  We begin by constructing a network $N_{a,b}$ and a run of the Ford-Fulkerson algorithm on $N_{a,b}$ which simulates the Euclidean algorithm on $a$ and $b$.  

 \begin{figure}[h]

\centering
\includegraphics[height=6cm]{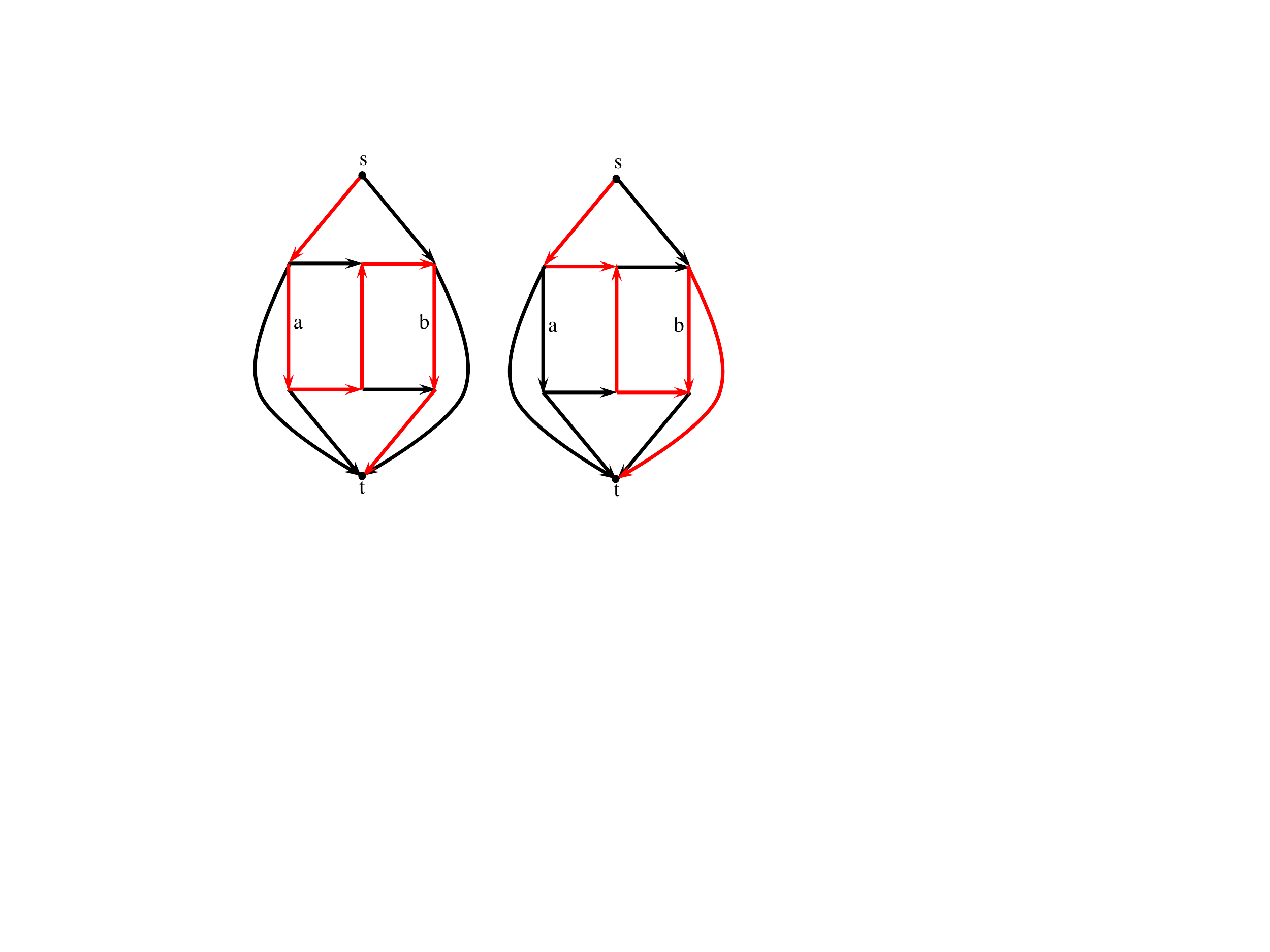}
\caption{ A network $N_{a,b}$ and two augmenting paths which allow for the subtraction of $b$ from $a$.}\label{euclidean}
\end{figure}

\begin{lemma}
Let $N_{a,b}$ be the network depicted in Figure \ref{euclidean}.  If $\frac{a}{b} \notin \mathbb{Q}$, then there is a run of the Ford-Fulkerson algorithm on $N_{a,b}$ with run-time at least $\omega$.
\end{lemma}

\begin{proof}
The labelled arcs of Figure \ref{euclidean} (denoted $e_a$ and $e_b$) have capacities $a$ and $b$ with $0< b \leq a$.   The other arcs have large capacity (which we will specify later).  Given a flow $f$ and an arc $e$, the \textit{residual capacity} of $e$, denoted $c'(e)$, is defined to be $c(e)-f(e)$.  Let $P_1$ and $P_2$ be the augmenting paths shown in red. Now, starting from the zero-flow, if we push flow along $P_1$ and $P_2$, then  $c'(e_a)$ becomes $a-b$ and $c'(e_b)$ is still $b$.  We continue
this process until $a_1:=a-n_1b \leq b$.   At this point, note that the value of the currrent flow is $2n_1b$ and that the roles of $a_1$ and $b$ are reversed.  

Let $P_1'$ and $P_2'$ be the reflections of $P_1$ and $P_2$ through the vertical line from $s$ to $t$.  
By next pushing flow along the augmenting paths $P_1'$ and $P_2'$, we can convert the residual capacity of $e_b$ to $b-a_1$.  We continue this process until $b_1:=b-m_1a_1 \leq a_1$. 
Note that the horizontal arcs are backward arcs of $P_1'$ and $P_2'$.  However, since $2m_1a_1 \leq 2b \leq 2n_1b$, there is enough flow along the horizontal arcs to perform these $2m_1$ augmentations. The roles of $a_1$ and $b_1$ have been reversed again and we continue inductively.  Therefore, this example ``computes'' the greatest common divisor of $a$ and $b$. If $\frac{a}{b} \notin \mathbb{Q}$, it will have run-time at least $\omega$, as required.  However, it remains to check that the total flow is still bounded after $\omega$ steps (so that we may specify the capacities).    
\begin{claim} \label{euclideanbound}
The total flow after $\omega$ steps of the Euclidean run on $N_{a,b}$ is at most $4(a+b)$.
\end{claim}
\begin{proof}
Let $a=a_0 > b=b_0 > a_1 > b_1 > \dots$ be the intermediate outputs of the Euclidean algorithm. Observe that for all $n \geq 0$, $a_{n+1} \leq \frac{1}{2}a_n$ and $b_{n+1} \leq \frac{1}{2}b_n$. Therefore, the total flow after $\omega$ steps is
at most $2\sum_{n=0}^{\infty} (a_n+b_n) \leq 4(a+b)$, as required. 
\end{proof}
Thus, for all arcs $e \notin \{e_a, e_b \}$, we can take $c(e)=4(a+b)$.  
\end{proof}

This example is quite robust in that a small perturbation of the arc capacities will almost certainly produce another non-terminating example.  We will now ``glue'' $k$ copies of $N_{a,b}$ in parallel to obtain a run of the Ford-Fulkerson algorithm with run-time at least $\omega^k$.  This proves Lemma \ref{lower}.  A key property of our gluing construction is that we can  ``recharge'' one copy of the Euclidean example using only two steps of the Euclidean algorithm of another copy.  

\begin{replemma}{lower}
For every $\ell \in \mathbb{N}$, there exists a network $N(\ell)$ on $\ell$ arcs and a run of the Ford-Fulkerson algorithm on $N(\ell)$ with run-time at least $\omega^{\left \lfloor{\ell \over  25}\right \rfloor}$. 
\end{replemma}

\begin{proof}
For each $k \in \mathbb{N}$, we define a network $N_{a,b}^k$ and show that there is a run of the Ford-Fulkerson algorithm on $N_{a,b}^k$ with run-time at least $\omega^k$. For simplicity, we only illustrate the case $k=2$; the general case is similar. Figure \ref{Glued} depicts the corresponding network for $k=2$.  The labelled arcs have capacities $a,b,a$, and $b$ respectively with $b<a$ incommensurate.  We denote these arcs as $\ell_1, \dots, \ell_4$ from left to right.   The remaining arcs have large capacities, which we will specify later. We denote the three middle vertical arcs as $e_1, e_2$ and $e_3$ from left to right. 

Note that there are ``left'' and ``right'' copies of $N_{a,b}$ sitting in $N_{a,b}^2$, depicted in green and blue, respectively, in Figure \ref{Glued}.
 \begin{figure}[h]

\centering
\includegraphics[height=11cm]{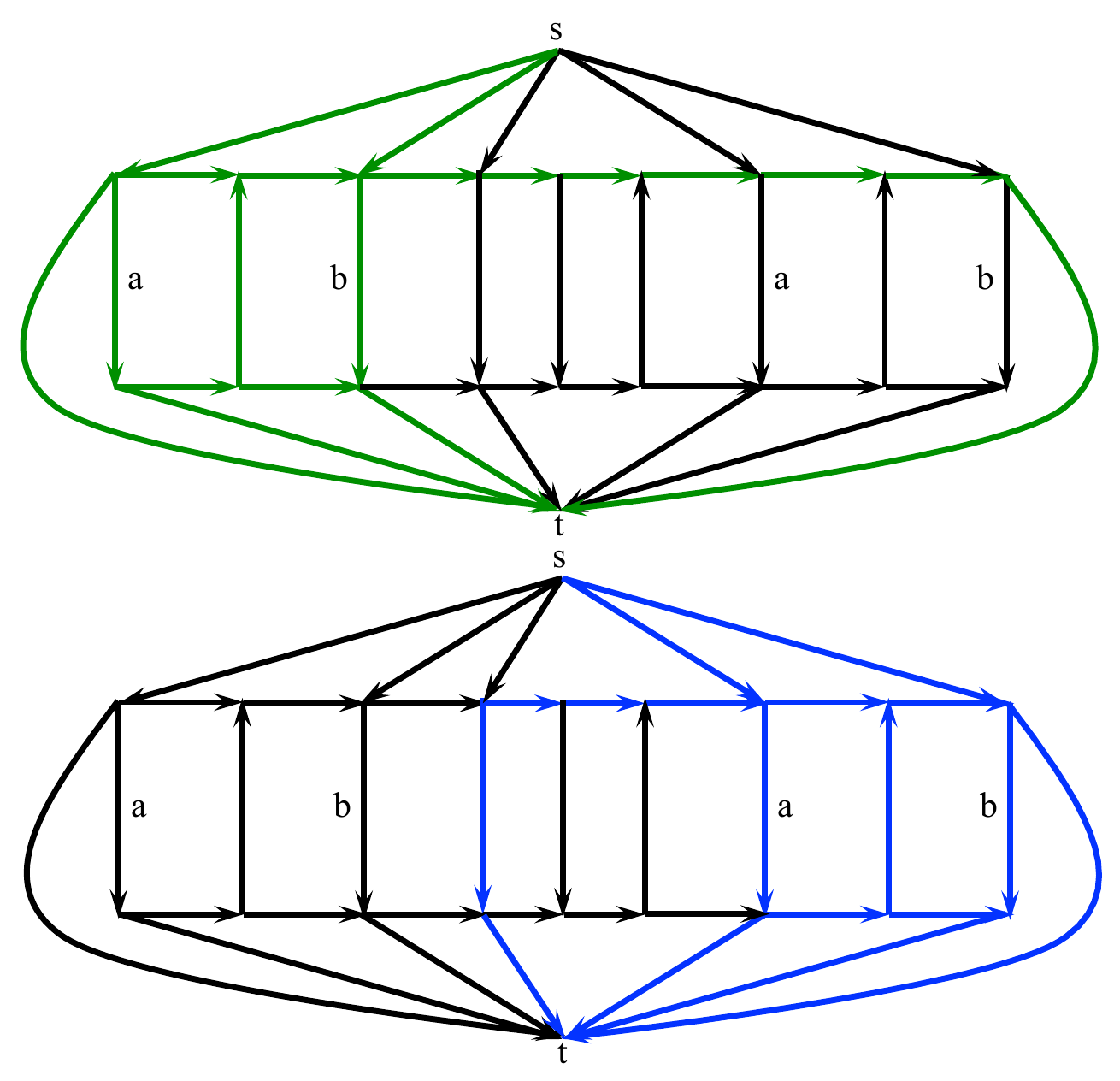}
\caption{A network $N_{a,b}^2$ (depicted twice) that admits a run of Ford-Fulkerson that takes $\omega^2$ iterations.  Left and right copies of $N_{a,b}$ in $N_{a,b}^2$ are shown in green and blue.}\label{Glued}
\end{figure}

Let $a=a_0 > b=b_0 > a_1 > b_1 > \dots$ be the intermediate outputs of the Euclidean algorithm for $(a,b)$.
We begin by running the Euclidean algorithm on the left copy of $N_{a,b}$.  This takes $\omega$ steps.   We then run two steps of the Euclidean algorithm on the right copy of $N_{a,b}$, obtaining residual capacities $a_1 > b_1$ for $\ell_3$ and $\ell_4$, respectively.  We next perform a recharging step (described below), and then re-run the (full) Euclidean algorithm on the left copy of $N_{a,b}$. We then run another two steps of the Euclidean algorithm on the right copy of $N_{a,b}$, obtaining residual capacities $a_2 > b_2$ and proceed iteratively.    

\textbf{The Recharging Step.}  At the beginning of the $n$th recharging step the residual capacities of $\ell_1, \ell_2, \ell_3$ and $\ell_4$ are $0, 0, a_n$ and $b_n$, respectively.  We first use the two augmenting paths in red  to recharge
the residual capacity of $\ell_1$ without drastically changing the rest of the network.   Observe that none of the middle arcs $e_1, e_2$ and $e_3$ are in the left or right copy of $N_{a,b}$. The  arcs $e_2$ and $e_3$ will always have zero-flow except during this recharging step.  By pushing flow along the two red augmenting paths,  $c'(e_1)$ becomes $a_n$ and $c'(e_2), c'(e_3)$ and $c'(e_4)$ are unchanged.   Note that both $e_2$ and $e_3$ have zero-flow after the recharging step, as claimed.  
Similarly, using two augmenting paths similar to the ones in red, we can recharge $c'(e_2)$ to $b_n$ without changing $c'(e_1), c'(e_3)$, and $c'(e_4)$.

One easily checks that there is always sufficient flow along backward arcs to run the Euclidean algorithm on the right copy of $N_{a,b}$ and to perform the recharging steps.  Furthermore, since $a$ and $b$ are incommensurate, $a_n$ and $b_n$ are also incommensurate for all $n$.  Therefore, this run of Ford-Fulkerson on $N_{a,b}^2$ requires at least $\omega^2$ steps.  By Claim \ref{euclideanbound}, the total flow after $\omega^2$ steps is at most
\[
4(a+b) + 4\sum_{n=0}^{\infty} (a_n+b_n) + 2\sum_{n=1}^{\infty} (a_n+b_n) \leq 14(a+b),
\]
where the first term corresponds to the Euclidean run on the right copy of $N_{a,b}$, the second term corresponds to
all the Euclidean runs on the sequence of left networks $N_{a_0,b_0}, N_{a_1,b_1}, \dots$, and the third term corresponds to the
total flow produced by the recharging steps.  

We can therefore take $c(e)=14(a+b)$ for all $e \notin \{\ell_1, \dots, \ell_4\}$. The proof is complete as it is straightforward to check that $N_{a,b}^k$ contains at most $25k$ edges for all $k$.  
\end{proof}

 \begin{figure}[h]

\centering
\includegraphics[height=11cm]{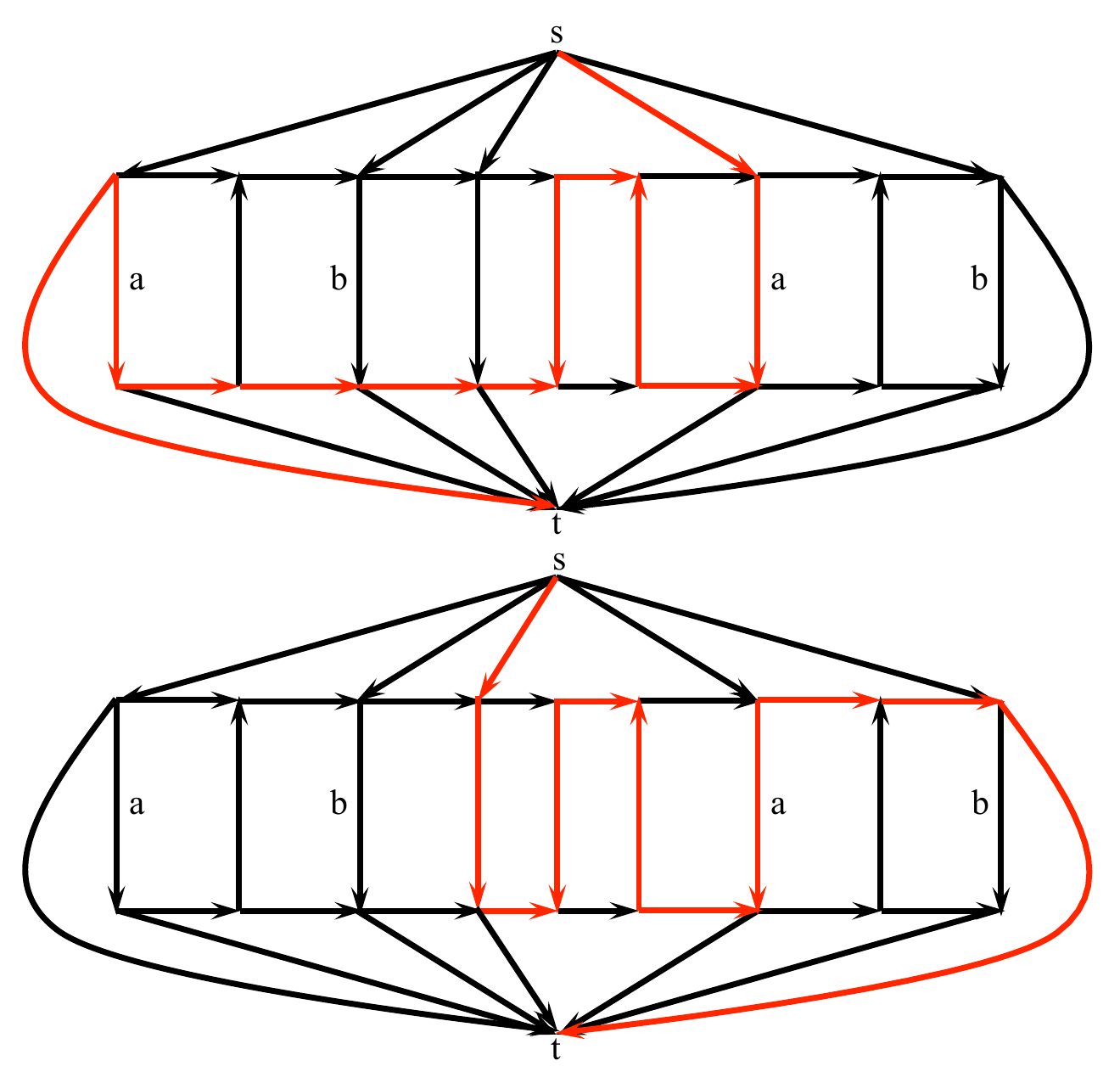}
\caption{A sequence of two augmenting paths that recharges the far left vertical edge is shown in red.}\label{parallel}
\end{figure}

Note that our example for the lowerbound starts with the zero-flow, while our proof for the upperbound is valid starting with any initial flow.

\section{Chip-firing and a Restricted Duality}

The first author \cite{backman2014infinite} proved the non-termination, and investigated the transfinite behavior, of a certain greedy chip-firing algorithm on metric graphs.  Many of the results obtained and the arguments employed in that paper are similar to those which appear in this article.  In this section we attempt to shed some light on the similarity between these two works by showing how in a restricted setting, a variant of the chip-firing algorithm previously investigated is planar dual to the Ford-Fulkerson algorithm.  We wish to emphasize that the duality presented in this section is not strong enough to imply that the results of either paper follow from the other.  We begin with a review of chip-firing.  We have attempted to make our introduction elementary and historically motivated, but nearly all proofs are omitted.  There are many references provided for the interested reader.

Chip-firing is a simple game played on the vertices of a graph, which has been independently discovered in several different communities \cite{bak1988self, bjorner1991chip, engel1975probabilistic, mosesian1972strongly}.  A \textit{configuration of chips} on the vertices of a graph is a function from the vertices to the integers.  We imagine each integer as describing some number of poker chips sitting at the corresponding vertex.  The vertices with a negative number of chips are said to be \emph{in debt}.  Given a chip configuration $D$, a vertex $v$ {\it fires} by sending a chip to each of its neighbors and losing its degree number of chips in the process so that the total number of chips is conserved.  This setup gives a combinatorial language for understanding the integer translates of the lattice generated by the columns of the Laplacian matrix.  There has been a recent explosion of interest in chip-firing as the natural language for developing a theory of divisors on graphs \cite{baker2007riemann, nagnibeda1997roland} and abstract tropical curves \cite{gathmann2008riemann, mikhalkin2006tropical, ranksofdivisors} analogous to the classic theory for algebraic curves.  An abstract tropical curve is essentially a metric graph; an edge weighted graph where each edge is isometric to an interval of length equal to the associated weight.  A \emph{chip configuration} on a metric graph is a function from the points of the metric graph to the integers with finite support.  In the discrete case, if a set of vertices $U$ fires, a chip is sent across each edge in $(U,U^c)$.  For a metric graph, we define the basic chip-firing moves by taking an edge cut of length $\epsilon$ and sending a chip across each segment in the cut.

  In keeping with the language of (tropical) algebraic geometry, we may refer to a chip configuration as a {\it divisor}, and we say that two divisors are {\it linearly equivalent} if we can get from one to the other by a sequence of chip-firing moves.  The main combinatorial tool for studying chip-firing on metric graphs is the {\it $q$-reduced divisor}, whose discrete analogue is also known as a $G$-parking function \cite{postnikov2004trees} and is dual to the recurrent configurations in the sandpile model \cite{baker2007riemann} \cite{biggs1999chip}.  We say that the number of chips in a divisor $D$ is the {\it degree of D} and write ${\rm deg}(D)$ for this quantity.
  
  \begin{definition}
  A divisor $D$ is \emph{$q$-reduced} if $D(p) \geq 0$ for all $p \neq q$ and for all $A \subset V(G) \setminus{q}$, firing the set $A$ causes some vertex to be sent into debt.  
    \end{definition}
    
    The fundamental theorem about $q$-reduced divisors is the following.
    
    \begin{theorem}\cite{baker2007riemann}\cite{gathmann2008riemann}\cite{ranksofdivisors}\cite{mikhalkin2006tropical}
  Each divisor $D$ is linearly equivalent to a unique $q$-reduced divisor.  Moreover, $D$ is linearly equivalent to a nonnegative divisor if and only if the equivalent $q$-reduced divisor is nonnegative.  
  \end{theorem}

In the discrete setting, Dhar \cite{dhar1990self} provided an efficient method for testing whether a divisor is $q$-reduced. Imagine that each vertex has $D(v)$ firefighters present.  A fire is started at the root $q$ which spreads through the graph, but is prevented from burning through the vertices by the firefighters present.  When the fire approaches a vertex $v$ from more than $D(v)$ directions, the firefighters are overpowered and the fire passes through the vertex.  Dhar observed that $D$ is reduced if and only if the fire consumes the entire graph.  If $D$ is not reduced, then the set of vertices which are not burnt can be fired (not to be confused with burned) simultaneously without sending any vertex into debt, thus bringing the divisor closer to being reduced.  It is easy to check that if a divisor $D$ is not reduced and $U_1$ and $U_2$ are two sets of vertices which can be fired without sending any vertex into debt, then we can also fire $U_1 \cup U_2$ without sending any vertex into debt.  Thus, the set of ``fireable'' vertices forms a join semilattice, and Dhar's algorithm finds the unique maximum element in this semilattice. 

Given a divisor on a metric graph which is not $q$-reduced, one may repeatedly perform maximal firings towards $q$. If this process terminates, we arrive at the unique equivalent $q$-reduced divisor.  Unlike the discrete case, it is not clear {\it a priori} whether this process will terminate.  Luo \cite{luo2011rank} introduced a metric version of Dhar's algorithm and showed that this gives a finite method for computing the associated $q$-reduced divisor. The question of whether the greedy reduction method also terminates in finite time was left open. The first author \cite{backman2014infinite} demonstrated that the greedy reduction algorithm need not terminate, but as is the case with the Ford-Fulkerson algorithm, the greedy reduction algorithm always has a well-defined limit.  Thus, the reduction algorithm was interpreted as a transfinite algorithm and its running time was analyzed using ordinal numbers.  The main result of \cite{backman2014infinite} is the following.

\begin{theorem}\cite{backman2014infinite}
The worst case running time for the greedy reduction algorithm of a divisor on a metric graph is $\omega^{\Theta({\rm deg}(D))}$.
\end{theorem}

Given an orientation $\O$ of a discrete graph, we can associate a chip configuration $D_{\O}$ by taking the indegree minus one at each vertex.  Our connection between Ford-Fulkerson and the greedy reduction algorithm is derived from the following well-known relationship between $q$-reduced divisors and acyclic graph orientations.  

\begin{theorem}
A divisor $D$ with $D(q)=-1$ is a maximal $q$-reduced divisor if and only if $D = D_{\O}$ where $\O$ is an acyclic orientation with a unique source at $q$.
\end{theorem}
If we take a directed cut $(U,U^c)$ in $\O$ and reverse all of the edges, we obtain a new orientation $\O'$ such that the associated chip configuration $D_{\O'}$ is obtained from $D_{\O}$ by firing all of the vertices in $U$. Every acyclic orientation is equivalent via source reversals (or more generally directed cut reversals) to a unique orientation with a unique source at $q$.  This object can be obtained by arbitrarily flipping sources other than $q$ (or more generally directed cuts orientated towards $q$).  

It is easy to check that an acyclic orientation $\O$ has a unique source $q$ if and only if every other vertex is reachable from $q$ by a directed path.  This suggests the following efficient method to obtain the unique $q$-connected orientation equivalent to $\O$ by cut reversals. We first perform directed search from $q$.  If each vertex is reachable by some directed path, we are done. Otherwise, let $U$ be the set of vertices reachable from $q$.  The edge cut $(U,U^c)$ is oriented towards $q$, and we can flip this directed cut, bringing us closer to the desired orientation.  As is the case with chip configurations, given an orientation $\O$ and a root $q$, we can associate the collection of sets of vertices $U$, such that $(U,U^c)$ is a directed cut in $\O$ with $q \in U^c$, and this collection forms a join semilattice.  The cut which we obtain by search corresponds to the unique maximum element in this join semilattice.  In this sense, the search-based algorithm just described may naturally be viewed as a variant of Dhar's burning algorithm for graph orientations.  For the remainder of this section we will call this algorithm the {\it pseudo Dhar's algorithm}, which is not to be confused with the \emph{oriented Dhar's algorithm} appearing in \cite{backman2014riemann} (see Section 7 of that paper for a different connection between chip-firing and network flows).

So as to not mislead the reader, we also highlight some differences between Dhar's algorithm and the pseudo Dhar's algorithm.  The following distinctions hold for both the discrete and metric versions of these two algorithms.  The first difference is Dhar's algorithm proceeds by firing sets of vertices not including $q$, whereas the cut reversals described in the oriented version correspond to borrowings (the inverse of firings) by sets of vertices not including $q$. The second difference is that Dhar's algorithm for reducing divisors presumes that the divisor in question is nonnegative away from $q$, but for acyclic orientations, the associated divisor is not effective away from $q$.  In fact, the oriented variant of Dhar's algorithm applied to the acyclic orientation produces an acyclic orientation with an associated divisor which is nonnegative away from $q$ precisely when $q$ is the unique sink, and this occurs when the divisor is $q$-reduced.

A metric graph (orientation) can be considered as a limit of (orientations of) discrete graphs under repeated subdivisions of the edges, where the ratios of the induced path lengths converge to the desired ratios of lengths.  On the other hand, networks with real capacities may be viewed as a limit of networks having all edge capacities 1 where we continue adding parallel directed edges so that the ratios of the number of parallel edges converge to the desired ratios of the capacities.  For plane graphs, these limits are planar dual as replacing an edge with $k$ parallel edges is dual to replacing the dual edge with a path of length $k$.  See Figure \ref{networkmetric}.

   \begin{figure}[h]
  \centering
\includegraphics[height=5cm]{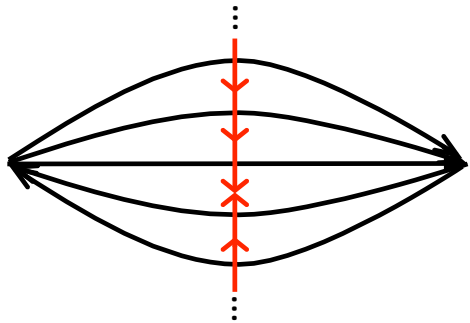}
\caption{The duality between planar parallel directed edges and planar induced paths of directed edges}\label{networkmetric}
\end{figure}

We note that while the termination of Luo's metric version of Dhar's algorithm is not obvious, the termination of the metric version of the pseudo Dhar's algorithm is.  Every time we are stuck with some cut oriented towards $q$ and we push this cut back as far as we can, we reach some new vertex.  Hence, this process will terminate in at most $|V|$ steps.  It is now natural to ask (analogous to the greedy reduction algorithm), if we reverse maximal cuts directed towards $q$ arbitrarily to obtain the unique equivalent $q$-connected acyclic orientation, whether this process will terminate in finite time.  We will call this method the {\it greedy cut reversals algorithm}.

A {\it $s$-$t$ planar network} is a planar network such that $s$ and $t$ belong to the same face.  We will call an acyclic $s$-$t$ planar network such that every edge belongs to a directed path from $s$ to $t$ a {\it simple $s$-$t$ planar network}.   Note that the networks appearing in our lower bound construction are of this form.  Given a planar network $N$ and an orientation of the plane, we define $\overline{N}$ to be the planar dual network.  

Berge \cite{berge} investigated $s$-$t$ planar networks and showed that a variant of Ford-Fulkerson exists which necessarily terminates in finite time and has better worst case running-time than the Edmonds-Karp variant of Ford-Fulkerson.  His method was elegant; when performing the Ford-Fulkerson algorithm, always choose the uppermost augmenting path.  We are now ready to describe the main observation of this section.

\begin{observation}
Given a simple $s$-$t$ planar network $N$, add an arc $e=(t,s)$ of very large capacity which passes above $N$.  The Ford-Fulkerson algorithm on $N$ is dual to the greedy cut reversal algorithm on the metric graph orientation $\overline {N \cup \{e\}}$.  Moreover, Berge's algorithm for $N$ is planar dual to the pseudo Dhar's algorithm for $\overline {N \cup \{v\}}$.  
\end{observation}

We note that there is some ambiguity in this duality because the capacity of $e$ has not been specified. However, since $v$ was assumed to have very large capacity, the dual directed edge will never affect a run of the pseudo Dhar's algorithm.  Hence, the dual of $N$ is essentially well-defined.  By adding $\{e\}$ to $N$, we obtain an auxiliary directed graph which is strongly connected.  Now, augmenting flows in our original network can be extended via $\{e\}$ to reversals of weighted directed cycles containing the unique clockwise face $q$ containing $s$ and $t$.  Thus, the general Ford-Fulkerson algorithm may be interpreted as greedily performing directed weighted cycle reversals until all of the cycles enclosing $q$ are oriented counterclockwise.  The planar dual $\overline {N \cup \{e\}}$ is an acyclic metric graph orientation and the weighted directed cycle reversals in $N$ (which necessarily contain $e$) are planar dual to directed cut reversals in $\overline {N \cup \{e\}}$ toward $q$.  One may check that there are no clockwise directed cycles in our original digraph if and only if the planar dual of the auxiliary network is $q$-connected, where as an abuse of notation we have identified the face $q$ and the corresponding dual vertex.  Berge's uppermost augmenting path extends via $\{e\}$ to a directed cycle which is dual to the cut given by the pseudo Dhar's algorithm.  By taking the symmetric difference of this orientation with our original auxiliary directed graph and deleting $\{e\}$, we obtain a maximum flow in our original network.  

As an application of this duality, by taking the planar dual of the networks and augmenting paths described in Section \ref{sec:lowerbound}, we obtain acyclic orientations such that the greedy cut reversals algorithm takes at least $\omega^k$ steps to terminate.  See Figure \ref{dualeuc}.

  \begin{figure}[h]
  \centering
\includegraphics[height=8cm]{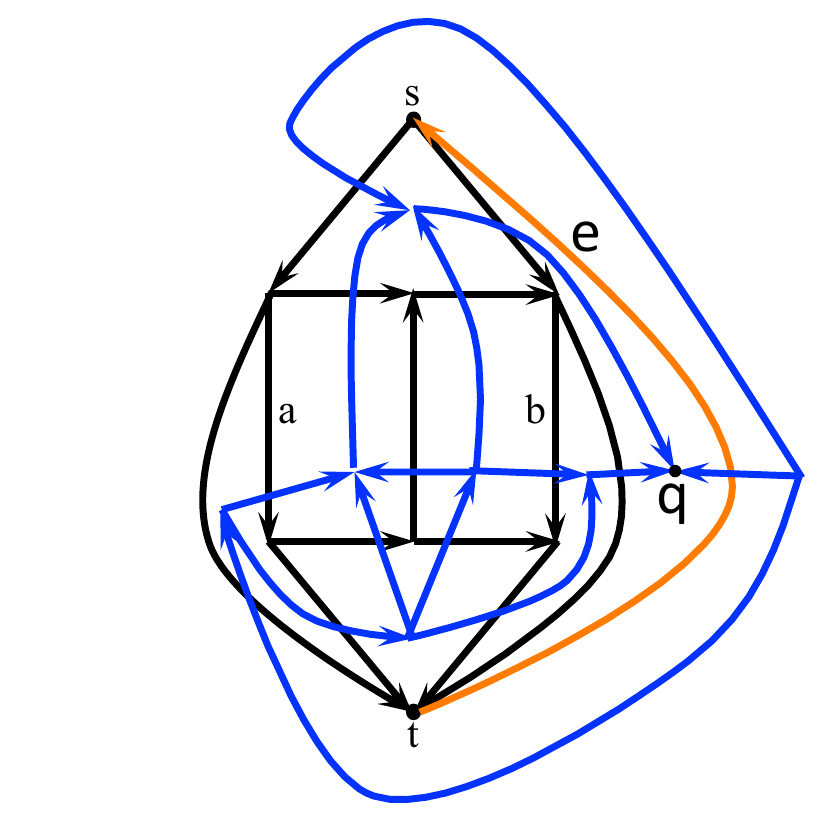}
\caption{The ayclic metric orientation dual to the Euclidean network.}\label{dualeuc}
\end{figure}

\

\textbf{Acknowledgements.} We would like to thank Matt Baker,  Diane Maclagan, and Vic Reiner for each noting the similarity between the non-termination of Ford-Fulkerson and the non-termination of metric chip-firing.

\bibliography{references}{}
\bibliographystyle{plain}

\end{document}